\documentclass[11pt]{amsart}
\usepackage{verbatim}
\usepackage{amssymb}
\usepackage{colonequals}
\usepackage{mathtools}
\usepackage[bookmarks,colorlinks,breaklinks]{hyperref}  
\usepackage[top=2cm, bottom=4.5cm, left=3cm, right=3cm]{geometry}
\hypersetup{linkcolor=blue,citecolor=blue,filecolor=blue,urlcolor=blue} 
\usepackage{bm}
\usepackage{color}
\usepackage{amsfonts,amsmath,amsthm,amssymb,amscd, hyperref}

\numberwithin{equation}{subsection}

\newtheorem{theorem}{Theorem}[section]
\newtheorem{corollary}[theorem]{Corollary}
\newtheorem{lemma}[theorem]{Lemma}
\newtheorem{proposition}[theorem]{Proposition}

\theoremstyle{definition}

\newtheorem{definition-theorem}[theorem]{Definition-Theorem}

\newtheorem{remark}[theorem]{Remark}

\theoremstyle{remark}

\newtheorem*{remark*}{Remark}

\usepackage{enumitem}
\setenumerate[1]{label=\textup{(\alph*)}}
\setenumerate[2]{label=\textup{(\roman*)}}

\renewcommand{\l}{\lambda}

\newcommand{\OO}{\mathcal{O}}

\newcommand{\ZZ}{\mathbb{Z}}

\newcommand{\QQbar}{\overline{\mathbb{Q}}}

\newcommand{\Z}{{\mathbb Z}}

\newcommand{\N}{{\mathbb N}}

\newcommand{\Q}{{\mathbb Q}}
\newcommand{\C}{{\mathbb C}}

\newcommand{\R}{{\mathbb R}}

\newcommand{\Qbar}{{\overline{\mathbb Q}}}

\newcommand{\Lbar}{{\overline{L}}}

\newcommand{\End}{{\rm End}}

\newcommand{\lra}{\longrightarrow}

\begin{document}
\title[Collision of orbits]{Collision of orbits on an elliptic surface}

\author{Dragos Ghioca}
\address{Department of Mathematics, University of British Columbia, Vancouver, BC V6T 1Z2}
\email{dghioca@math.ubc.ca}

\author{Negin Shadgar}
\address{Department of Mathematics, University of British Columbia, Vancouver, BC V6T 1Z2}
\email{negin@math.ubc.ca}



\subjclass[2020]{Primary 37P15; Secondary 11G35}

\keywords{elliptic surfaces; unlikely intersections; collision of orbits}

\begin{abstract}
Let $C$ be a smooth projective curve defined over $\Qbar$, let $\pi:\mathcal{E}\lra C$ be an elliptic surface and let  $\sigma_{P_1},\sigma_{P_2},\sigma_{Q}$ be sections of $\pi$ (corresponding to points $P_1,P_2,  Q$ of the generic fiber $E$ of $\mathcal{E}$).  We obtain a precise characterization, expressed solely in terms of the dynamical relations between the points $P_1,P_2,Q$ with respect to the endomorphism ring of $E$, so  that there exist infinitely many $\l\in C(\Qbar)$ with the property that for some nonzero integers $m_{1,\l},m_{2,\l}$, we have that $[m_{i,\l}](\sigma_{{P_{i}}}(\l))=\sigma_{Q}(\l)$ (for $i=1,2$) on the smooth fiber $E_\l$ of  $\mathcal{E}$. 
\end{abstract}

\maketitle

\section{Introduction}
\label{sec:intro}


\subsection{Notation and setup for our problem}

For an elliptic curve $E$ defined over a field $L$ of characteristic $0$, we denote by $\End(E)$ the ring of group endomorphisms of $E$ (defined over an algebraic closure $\Lbar$ of $L$). It is known (see \cite[Chapter~2]{Silverman}) that $\End(E)$ is isomorphic either with $\Z$ (when $E$ has no complex multiplication), or with an order in an imaginary quadratic number field (when $E$ has \emph{complex multiplication, also referred to as $CM$}). So, we identify $\End(E)$ with either $\Z$ (in the non-$CM$ case), or with an order $\OO$ in an imaginary quadratic number field (in the $CM$ case); more precisely, for each $n\in\Z$ (respectively, for each $\alpha\in\OO$), we denote by $[n]$ (resp.  $[\alpha]$) the corresponding endomorphism of $E$.

Throughout our paper, $C$ is a smooth, projective (irreducible) curve defined over $\Qbar$. Also, we consider $\pi:\mathcal{E}\lra C$ be a an elliptic surface, where the generic fiber of $\pi$ is an elliptic curve $E$ defined over $\Qbar(C)$, while for all but finitely many points $\l\in C(\Qbar)$, the fiber $E_\l:=\pi^{-1}(\{\l\})$ is an elliptic curve defined over $\Qbar$. We will often refer to $E_\l$ as the \emph{$\l$-fiber} of $\pi$. Each  section $\sigma$ of $\pi$ (i.e. a map $\sigma:C\lra\mathcal{E}$ such that $\pi\circ \sigma={\rm id}|_C$) gives
rise to a $\Qbar(C)$-rational point of $E$. Conversely, a point $P\in E(\Qbar(C))$ corresponds
to a section of $\pi$, denoted $\sigma_P$. For all but finitely many $\l\in C(\Qbar)$, the intersection of the
image of $\sigma_P$ in $\mathcal{E}$ with the smooth fiber $E_\l$ above $\l$ is a point $P_\l := \sigma_P(\l)$ on the elliptic curve $E_\l$; we will refer to $P_\l$ as the \emph{$\l$-specialization of the point $P$} and we drop the mention of the corresponding section $\sigma_P$. Furthermore, when working with a point $P_i\in E(\Qbar(C))$ (for some index $i$), we use the notation $P_{i,\l}$ for the corresponding $\l$-specialization of $P_i$.

We say that $\mathcal{E}$ is \emph{isotrivial} if the $j$-invariant (see \cite[Chapter~3]{Silverman}) of the generic fiber $E$ is a constant function (on $C$); for isotrivial elliptic surfaces $\mathcal{E}$, all smooth fibers of $\pi$ are isomorphic to $E$. If $\mathcal{E}$ is isotrivial, then at the expense of replacing $C$ by a finite cover $C'$ (defined over $\Qbar$) and also,  replacing $\mathcal{E}$ by  $\mathcal{E}\times_CC'$, we may assume $\mathcal{E}$ is a \emph{constant elliptic surface}. So, when $E$ is isotrivial, since our question is unchanged when replacing $C$ by a finite cover, we always assume there exists an elliptic curve $E_0$ defined over $\Qbar$ such that $\mathcal{E}=E_0\times_{Spec(\Qbar)}C$. In this case, we say that a section $\sigma_P$ of $\mathcal{E}$  is \emph{constant} if $P\in E_0(\Qbar)$. 

In our paper, we study the following question for  elliptic surfaces $\mathcal{E}\lra C$ defined over $\Qbar$; as before, we denote by $E$ the generic fiber of $\mathcal{E}$. So, with the above notation,  given two \emph{starting points} $P_1,P_2\in E(\Qbar(C))$ and one \emph{target point} $Q\in E(\Qbar(C))$ (all three of them non-torsion), we obtain necessary and sufficient conditions (in terms of the dynamics of $P_1,P_2,Q$ under the action of $\End(E)$)  so that $Q_\l$ is contained in both the cyclic group generated by  $P_{1,\l}$ and in the cyclic group generated by $P_{2,\l}$ inside $E_\l$.


\subsection{Our main result}

We prove the following statement.
\begin{theorem}
\label{thm:main}
Let $C$ be a smooth projective curve defined over $\Qbar$, let $\mathcal{E}\lra C$ be an elliptic surface and let  $P_1,P_2,  Q\in E(\Qbar(C))$ be non-torsion points on the generic fiber $E$ of $\pi$. Furthermore, if $\mathcal{E}$ is isotrivial, then not all three sections $\sigma_{P_1},\sigma_{P_2},\sigma_{Q}$ are constant. Then there exist infinitely many $\l\in C(\Qbar)$ such that for some nonzero integers $m_{1,\l},m_{2,\l}$, we have that $[m_{i,\l}](P_{i,\l})=Q_\l$ on the $\l$-fiber $E_\l$ of $\pi$ if and only if  at least one of the following conditions holds:
\begin{itemize}
\item[(A)] there exists a nonzero integer $k$ such that for some $i\in\{1,2\}$, we have $[k](P_i)=Q$ on $E$. 
\item[(B)] there exist nonzero integers $k_1$ and $k_2$ such that $[k_1](P_1)=[k_2](P_2)$ on $E$.
\item[(C)] $E$ has complex multiplication by some order $\OO$ in an imaginary quadratic number field and  there exist nonzero $\alpha_1,\alpha_2,\beta\in \OO$ with $\alpha_1/\alpha_2\notin\Q$ such that $[\alpha_1](P_1)=[\alpha_2](P_2)=[\beta](Q)$ on $E$.
\end{itemize}
\end{theorem}

Our Theorem~\ref{thm:main} strenghtens the main result of \cite[Theorem~1.1]{GHT-PJM}, which we state below as Theorem~\ref{thm:main_0}.

\begin{theorem}
\label{thm:main_0}
With the notation as in Theorem~\ref{thm:main} for $\mathcal{E},C,E,P_1,P_2,Q$, if there exist infinitely many $\l\in C(\Qbar)$ such that for some nonzero integers $m_{1,\l},m_{2,\l}$, we have that $[m_{i,\l}](P_{i,\l})=Q_\l$ (for $i=1,2$) on the $\l$-fiber $E_\l$ of $\pi$, then  at least one of the following two conditions holds:
\begin{itemize}
\item[(i)] there exists a nonzero integer $k$ such that for some $i\in\{1,2\}$, we have $[k](P_i)=Q$; 
\item[(ii)] there exist nontrivial endomorphisms $\alpha_1,\alpha_2 \in\End(E)$  such that $\alpha_1(P_1)=\alpha_2(P_2)$.
\end{itemize}
\end{theorem}

\begin{remark}
\label{rem:constant}
Here, we explain the hypothesis appearing in our Theorem~\ref{thm:main} regarding the points $P_1,P_2,Q$ not being constant when $\mathcal{E}$ is isotrivial. Indeed, working (after replacing $C$ by a suitable finite cover) with $\mathcal{E}=E_0\times_{{\rm Spec}(\Qbar)}C$, if $P_1,P_2,Q\in E_0(\Qbar)$ then the problem considered in Theorem~\ref{thm:main} loses its relevance since we have in this case (no matter what is the fiber of our elliptic surface)  the same dynamical system (in terms of the two starting points and the target point). In particular,  conditions~(A)-(C) stated as in Theorem~\ref{thm:main} no longer guarantee the existence of infinitely many $\l\in C(\Qbar)$ such that $Q_\l$ lives in the cyclic groups generated by $P_{i,\l}$ (for $i=1,2)$. 
\end{remark}

Our Theorem~\ref{thm:main} fits into the larger theme of unlikely intersections in arithmetic dynamics, as we will explain below.


\subsection{Unlikely intersections in arithmetic dynamics}

The following result was proved by Masser-Zannier (see \cite{M-Z-1,M-Z-2}); an alternative proof was derived by DeMarco-Mavraki (see \cite{D-M-1}).
\begin{theorem}
\label{thm:main_00}
Let $C$ be a smooth projective curve defined over $\Qbar$, let $\pi:\mathcal{E}\lra C$ be an elliptic surface and let  $P_1,P_2\in E(\Qbar(C))$ be non-torsion points  on the generic fiber $E$ of $\mathcal{E}$. Then there exist infinitely many $\l\in C(\Qbar)$ such that for some nonzero integers $m_{1,\l},m_{2,\l}$, we have that $[m_{i,\l}](P_{i,\l})=0$ (for $i=1,2$)  on the $\l$-fiber $E_\l$ of $\pi$, if and only if  there exist nontrivial endomorphisms $\alpha_1,\alpha_2\in \End(E)$  such that $\alpha_1(P_1)=\alpha_2(P_2)$.
\end{theorem}

So, Theorem~\ref{thm:main_00} shows that the infinite occurrence of the \emph{unlikely} event that \emph{both} $P_{1,\l}$ and $P_{2,\l}$ are torsion on $E_\l$ can only be explained by a global dynamical relation between the points $P_1$ and $P_2$; such a conclusion fits the pattern behind several other unlikely intersection problems, going back to classical questions in arithmetic geometry (for more details, see \cite{Umberto}). We also note that several extensions of the results from \cite{M-Z-1,M-Z-2} were obtained both in the case of elliptic surfaces (see \cite{Lau-1}) and also in the general case of semiabelian schemes (see \cite{Lau-2, Lau-3}).

Our Theorem~\ref{thm:main} can be viewed as an extension of Theorem~\ref{thm:main_00} since we replace the \emph{target} $0\in E$ from Theorem~\ref{thm:main_00} with an arbitrary point $Q\in E(\Qbar(C))$. Also, we note that for \emph{one} starting point $P_1$ and a target point $Q$, it is known, as proved in \cite[Theorem~A2]{Betti-4} (see also Lemma~\ref{lem:solutions}), that there exist infinitely many integers $n$ such that $[n](P_{1,\l})=Q_\l$. However, the fact that for some point $\l\in C(\Qbar)$, we have that $Q_\l$ is a multiple \emph{both} of $P_{1,\l}$ and of $P_{2,\l}$ is an \emph{unlikely event}, which can only be explained by a  global dynamical relation between the starting points $P_1,P_2$ and the target point $Q$ (as in conclusions~(A)-(C) from Theorem~\ref{thm:main}).

\begin{remark}
\label{rem:torsion}
It makes sense to ask that $Q$ is non-torsion in our Theorem~\ref{thm:main} since otherwise (at the expense of replacing $Q$ by a suitable multiple), we recover the setting from Theorem~\ref{thm:main_00}. Furthermore, once we work with a non-torsion target point $Q$ in Theorem~\ref{thm:main}, it is automatic to consider that also the starting points $P_1,P_2\in E(\Qbar(C))$ are non-torsion. Indeed, if there exists a positive integer $\ell$ such that $[\ell](P_1)=0$, then the existence of infinitely many $\lambda\in C(\Qbar)$ such that for some suitable nonzero integers $m_{1,\l}$, we have $[m_{1,\l}](P_{1,\l})=Q_\l$ yields that also $[\ell](Q_\l)=0$ for those $\l$'s. So, by the Pigeonhole Principle, there exists some torsion point $R$ of $E$ (such that $[\ell](R)=0$) with the property that for infinitely many $\l\in C(\Qbar)$, we have $R_\l=Q_\l$. This means that $R=Q$, contradicting the assumption that $Q$ is non-torsion. Therefore, the assumption that all three points $P_1,P_2,Q$ are non-torsion is justified for Theorem~\ref{thm:main}.
\end{remark}

The results of Masser-Zannier \cite{M-Z-1, M-Z-2} (see Theorem~\ref{thm:main_00} above) were the catalysts for an intense study in arithmetic dynamics of unlikely intersection questions (see the survey article \cite{survey} for a more in-depth discussion of it as one of  the  central questions in arithmetic dynamics and also, see the book \cite{Umberto} for a thorough study of the origins of the unlikely intersection questions in arithmetic geometry). In particular, the so-called question of \emph{simultaneously preperiodic points} (see also the survey \cite{G-survey}) generated a great deal of research in recent years (see \cite{Matt, Matt-2, GHT-ANT, 1} and the references therein). So, consider a family $F_\l(z)$ of polynomials parameterized by the $\C$-points $\l$  on some projective smooth curve $C$; also, let $\mathbf{a},\mathbf{b}:C\lra\mathbb{P}^1$. Then one obtains (see \cite{F-G}) a complete characterization (based on the dynamics of $\mathbf{a},\mathbf{b}$ under the action of the family of maps $\{F_\l\}_{\l\in C}$) for the following unlikely scenario to happen: there exist infinitely many $\l\in C(\C)$ such that both $\mathbf{a}(\l)$ and $\mathbf{b}(\l)$ are preperiodic under the action of $F_\l(z)$.  The general principle behind this question involving simultaneously preperiodic points is that the aforementioned  infinite occurrence of the {unlikely event} that both the orbits of $\mathbf{a}(\l)$ and of $\mathbf{b}(\l)$ are finite under the action of $F_\l(z)$ is explained only due to a \emph{global dynamical relation} between the two starting points $\mathbf{a}$ and $\mathbf{b}$ under the action of $F_\eta(z)$, where $\eta$ is the generic point on $C$.

It is natural to extend the above study of unlikely intersections to the following scenario. With the above notation for the family of maps $F_\l$ and for the \emph{two starting points} $\mathbf{a},\mathbf{b}$, one considers also a \emph{target point} $\mathbf{c}:C\lra\mathbb{P}^1$ and then asks for a precise characterization of when the following unlikely scenario occurs: there exist infinitely many points $\l\in C(\C)$ such that $\mathbf{c}(\l)$ lives in the forward orbits of both $\mathbf{a}(\l)$ and of $\mathbf{b}(\l)$ under the action of $F_\l(z)$ (i.e., there exist positive integers $m$ and $n$ such that $F_\l^m(\mathbf{a}(\l))=\mathbf{c}(\l)$ and $F_\l^n(\mathbf{b}(\l))=\mathbf{c}(\l)$). This problem is often coined as the \emph{collision of orbits question} (see \cite{A-G, G-JNT}). Using the same strategy patented in \cite{Matt} (and further refined in the subsequent papers \cite{Matt-2, GHT-ANT}), one can obtain global conditions which are \emph{necessary} for this collision of orbits question; however, obtaining precise \emph{necessary and sufficient} conditions is much more delicate compared to the previous question of simultaneously preperiodic points (see the discussion from both \cite[Remark~2.8]{A-G} and \cite[Section~2]{G-JNT}). Our Theorem~\ref{thm:main} provides a complete answer to a variant of this collision of orbits question in the context of algebraic families of elliptic curves $E_\l$, this time replacing the orbit of a point under a polynomial (or rational) map with the orbit of the  corresponding point on $E_\l$ under the action of $\Z$, seen as a subring of $\End(E_\l)$.


\subsection{Previous results and further connections to our work}
\label{subsec:previous}
A variant of the problem of collision of orbits is the following question studied by Hsia and Tucker \cite{HT}. Given polynomials $\mathbf{a},\mathbf{b},\mathbf{c}\in\C[z]$ (whose compositional powers $\mathbf{a}^m,\mathbf{b}^n,\mathbf{c}^k$ for any positive integers $m,n,k$ are all distinct), then \cite[Theorem~1]{HT} establishes that there exist finitely many $\l\in\C$ such that
\begin{equation}
\label{eq:gcd}
(z-\l)\mid \gcd\left(\mathbf{a}^m(z)-\mathbf{c}(z),\mathbf{b}^n(z)-\mathbf{c}(z)\right) \text{ for some $m,n\in\N$.}
\end{equation}
So, at its heart, the problem of collision of orbits is a question regarding the greatest common divisors of certain polynomials (see also the discussion from \cite[Section~9]{A-G}). This has its origins in the problem studied by 
Ailon-Rudnick \cite{A-R}, who proved that if $\mathbf{a},\mathbf{b}\in \C[z]$ are multiplicatively independent, then there exists a nonzero polynomial $\mathbf{d}\in\C[z]$ such that 
$$\gcd\left(\mathbf{a}(z)^k-1,\mathbf{b}(z)^k-1\right) \mid \mathbf{d}(z)\text{ for all $k\in\N$,}$$
where $\mathbf{a}(z)^k,\mathbf{b}(z)^k$ are the (usual) $k$-th powers of the polynomials $\mathbf{a}(z),\mathbf{b}(z)$.
 
In turn, the result of Ailon-Rudnick was motivated by the work of Bugeaud-Corvaja-Zannier \cite{BCZ} who established an upper bound for $\gcd(a^k - 1,b^k - 1)$ (as $k$ varies in $\N$) for given $a,b\in\Q$. We also mention that this problem of bounding
the greatest common divisor has been studied in several other directions as well: for elements close to $S$-units (see \cite{C-Z-2, Luca}), for polynomials defined over fields of positive characteristic (see \cite{GHT-NJM}), and also, for elliptic divisibility sequences (see \cite{Silverman-3}). Furthermore, one of the open questions raised in \cite{Silverman-3} was answered in the affirmative in \cite[Theorem~1.1]{GHT-PJM} (stated above as  Theorem~\ref{thm:main_0}). Also, we mention extensions of the results from \cite{GHT-PJM} were obtained in the more general case of semiabelian schemes (see \cite{Lau-4}).


\subsection{Plan for our paper}

Most of the work for obtaining the direct implication in Theorem~\ref{thm:main} was done in \cite[Theorem~1.1]{GHT-PJM} (see Theorem~\ref{thm:main_0} above). However, in the case when the elliptic surface admits complex multiplication, we need to sharpen the statement of Theorem~\ref{thm:main_0}~(ii) to deliver conclusion~(C) in our Theorem~\ref{thm:main}. This part is done in Section~\ref{sec:direct} (see Proposition~\ref{prop:direct}).

We prove the converse implication from Theorem~\ref{thm:main} by considering separately each one of the three conditions~(A),~(B)~and~(C) (see Lemmas~\ref{lem:A}~and~\ref{lem:B} and Proposition~\ref{prop:converse_main}) and show that in each case there exist infinitely many points $\l\in C(\Qbar)$ such that $Q_\l$ is (an integer) multiple of both $P_{1,\l}$ and of $P_{2,\l}$. Once again, when the elliptic surface admits complex multiplication (i.e., condition~(C) holds), we need a more careful analysis; in particular, we need to employ some basic facts regarding the arithmetic of orders in an imaginary quadratic number field (for more details, see Section~\ref{sec:CM}).

\medskip

{\bf Acknowledgments.} First of all, we warmly thank  Umberto Zannier for his always generous sharing of beautiful mathematical ideas, which led to writing this paper. We are grateful to the anonymous referee for their useful comments and suggestions. We also thank Laura Capuano, Simone Coccia, Patrick Ingram and Matteo Verzobio for helpful conversations. The research of the first author was partially supported by a Discovery Grant from the NSERC.


\section{Arithmetic properties for an elliptic curve with complex multiplication}
\label{sec:CM}

Let $E$ be a $CM$ elliptic curve defined over a field $L$ of characteristic $0$;    let $\OO$ be its  endomorphism ring. We identify $\OO$ with an order in some imaginary quadratic number field; i.e, $\OO=\Z[\theta]$, where (for some $f\in\N$ and some square-free positive integer $D$) we have: 
\begin{equation}
\label{eq:theta}
\theta=\left\{\begin{array}{ccc}
f\cdot i\sqrt{D} & \text{ if } & D\not\equiv 3\pmod{4}\\
f\cdot \frac{1+i\sqrt{D}}{2} & \text{ if } & D\equiv 3\pmod{4}
\end{array}\right.
\end{equation}

We recall the definition of the norm $N(\cdot )$ in the quadratic number field $\Q(i\sqrt{D})$; in particular, for any $a+b\theta\in \OO$, we have 
\begin{equation}
\label{eq:norm}
N(a+b\theta)=(a+b\theta)\cdot (a+b\bar{\theta}),
\end{equation}
where $\bar{\gamma}$ is simply the complex conjugate of any $\gamma\in\Q(i\sqrt{D})$.

In our proof, we will employ few easy facts (see Lemmas~\ref{lem:arithm_1}~and~\ref{lem:arithm_2}) regarding the arithmetic of the ring $\OO$. 

\begin{lemma}
\label{lem:arithm_1}
With $D,\theta$ as in equation~\eqref{eq:theta}, let also $a$ be an odd integer. Then there exist $m,r,s\in\Z$ such that
\begin{equation}
\label{eq:30}
m-\theta=(a+4\theta)\cdot (r+s\theta).
\end{equation}
\end{lemma}

Lemma~\ref{lem:arithm_1} can be proven directly, but it is also an immediate  consequence of the following more general result; we warmly thank the referee for suggesting Proposition~\ref{prop:more_general}.

\begin{proposition}
\label{prop:more_general}
Let $\alpha\in \OO$ with the property that there exists no rational prime divisor of $\alpha$. Then the composite map $\Z\hookrightarrow\OO\to \OO/\alpha\cdot \OO$ is surjective.
\end{proposition}

\begin{remark}
In Lemma~\ref{lem:arithm_1}, since $a$ is odd, we see that $\alpha=a+4\theta\in\OO$ satisfies the hypotheses of Proposition~\ref{prop:more_general} and therefore, there exists $m\in\ZZ$ such that $m-\theta\in \alpha\cdot \OO$, i.e., there exist also $r,s\in\ZZ$ (depending on $m$) such that equation~\eqref{eq:30} holds. 
\end{remark}

\begin{proof}[Proof of Proposition~\ref{prop:more_general}.]
The hypotheses of Proposition~\ref{prop:more_general} yields that there exist coprime integers $a$ and $b$ such that $\alpha=a+b\theta$. In order to ontain the desired conclusion in Proposition~\ref{prop:more_general}, it suffices to show there exist $m,x,y\in\ZZ$ such that
\begin{equation}
\label{eq:300}
m-\theta=(a+b\theta)\cdot (x+y\theta).
\end{equation} 
If $\theta=f\cdot i\sqrt{D}$ (see equation~\eqref{eq:theta}), then using that $\theta^2=-f^2D$ along with the fact that $\gcd(a,b)=1$, we immediately see that there exist suitable integers $m,x,y$ as in  equation~\eqref{eq:300}. 

Now, if $\theta=f\cdot \left(\frac{1+i\sqrt{D}}{2}\right)$ (see again equation~\eqref{eq:theta}), we obtain that 
\begin{equation}
\label{eq:301}
\theta^2-2f\theta+\frac{f^2(1+D)}{4}=0\text{ (note that $D\equiv 3\pmod{4}$ in this case).}
\end{equation}
Then there exist suitable integers $m,x,y$ as in equation~\eqref{eq:300}, as long as there exist integers $x$ and $y$ such that:
\begin{equation}
\label{eq:302}
1=ay+bx+2fby=ay+b\cdot (x+2fy).
\end{equation}
Since $\gcd(a,b)=1$, there exist integers $y_0$ and $x_0$ such that $1=ay_0+bx_0$ and therefore, letting $y=y_0$ and $x=x_0-2fy_0$, we see that equation~\eqref{eq:302} is verified (and thus, equation~\eqref{eq:300} holds). This concludes our proof of Proposition~\ref{prop:more_general}.
\end{proof}

Lemma~\ref{lem:arithm_1} allows us to derive the following easy consequence.
\begin{corollary}
\label{cor:action}
Let $a$ be an odd integer, let $Q$ be a torsion point of $E$ satisfying $[a+4\theta](Q)=0$ and let $U$ be the cyclic (finite) subgroup of $E$ generated by $Q$. Then for each  $\alpha\in\OO$, the following statements hold:
\begin{enumerate}
\item[(i)]  $[\alpha](Q)\in U$, i.e., $[\alpha]$ induces a natural group homomorphism $\widetilde{[\alpha]}$ on $U$ given by $\widetilde{[\alpha]}(Q)=[\alpha](Q)$.
\item[(ii)] if $\gcd\left(N\left(a+4\theta\right), N(\alpha)\right)=1$, then $\widetilde{[\alpha]}:U\lra U$ is a group automorphism.
\end{enumerate}
\end{corollary}

\begin{proof}
Using Lemma~\ref{lem:arithm_1}, we know there exist $m,r,s\in\Z$ such that 
\begin{equation}
\label{eq:41}
m-\theta=(a+4\theta)\cdot (r+s\theta).
\end{equation}
Equation~\eqref{eq:41} (along with the fact $[a+4\theta](Q)=0$) shows that $[\theta](Q)=[m](Q)$, i.e., $[\theta]$ induces the natural group homomorphism $\widetilde{[\theta]}$ on $U$ given by 
\begin{equation}
\label{eq:42}
\widetilde{[\theta]}(Q)=[m](Q).
\end{equation}
In particular, this means that writing $\alpha:=c+d\theta$ (for some $c,d\in\Z$), then $[\alpha]$ induces a group homomorphism $\widetilde{[\alpha]}$ on $U$. More precisely, equation~\eqref{eq:42} yields that 
\begin{equation}
\label{eq:43}
[\alpha](Q)=[c+dm](Q).
\end{equation}  
Equation~\eqref{eq:43} proves conclusion~(i) of Corollary~\ref{cor:action}.

Now, assuming that $\gcd\left(N(a+4\theta),N(\alpha)\right)=1$, we get that $[N(\alpha)]$ induces a natural group automorphism of the cyclic group $U$ (whose order divides $N(a+4\theta)$). Part~(i) shows that $[\alpha]$ and $[\bar{\alpha}]$  induce the group homomorphisms $\widetilde{[\alpha]}$ and $\widetilde{[\bar{\alpha}]}$ on $U$ for which
$$
\widetilde{[\alpha]}\circ \widetilde{[\bar{\alpha}]}=\widetilde{[N(\alpha)]} \text{ is a group automorphism for $U$.}$$
Therefore, $\widetilde{[\alpha]}$ is a group automorphism of $U$, as desired for part~(ii) of the conclusion in Corollary~\ref{cor:action}. 
\end{proof}

The next result shows that we can always find (infinitely many) elements in $\OO$ of the form $a+4\theta$ whose norms are coprime with a given positive integer. In particular, this will allow us to employ Corollary~\ref{cor:action} in our proof of Proposition~\ref{prop:converse_main} (which will deliver the converse implication in Theorem~\ref{thm:main} for $CM$ elliptic curves).

\begin{lemma}
\label{lem:arithm_2}
Let $M$ be a positive integer. Then there exists $\ell\in\{1,\dots, M\}$ such that for each integer $a\in\Z$ satisfying
\begin{equation}
\label{eq:congr}
a\equiv 2\ell-1\pmod{2M},
\end{equation}
we have $\gcd\left(N(a+4\theta),2M\right)=1$.
\end{lemma}

\begin{proof}
We split our analysis based on the two cases appearing for $\theta$ as in the formulas~\eqref{eq:theta}. The case when $\theta=f\cdot i\sqrt{D}$ is easier than the case when  $\theta=f\cdot \frac{1+i\sqrt{D}}{2}$ (and $D\equiv 3\pmod{4}$); therefore, we will focus below on the latter case (especially, since the analysis for the former case is similar). So, using equation~\eqref{eq:norm}, we have that
\begin{equation}
\label{eq:norm_equation}
N\left(a+4\theta\right)=N\left(a+2f+i\cdot 2f\sqrt{D}\right)=(a+2f)^2+4f^2D.
\end{equation}
For each prime $p$ dividing $2M$, we have two cases, depending on whether $p\mid 2f\cdot D$, or not; either way, we can always find some residue class $r_p\in\{0,1,\dots,p-1\}$ such that if $a\equiv r_p\pmod{p}$, then $p\nmid N(a+4\theta)$ (see equation~\eqref{eq:norm_equation}). Noting that $r_2=1$ (i.e., $a$ has to be odd so that $N(a+4\theta)$ is odd), we obtain the desired conclusion from Lemma~\ref{lem:arithm_2}.    
\end{proof}


\section{Proof of the direct implication in Theorem~\ref{thm:main}}
\label{sec:direct}

In this Section, we prove the following result.
\begin{proposition}
\label{prop:direct}
Let $C$ be a smooth projective curve defined over $\Qbar$, let $\mathcal{E}\lra C$ be an elliptic surface and let  $P_1,P_2,  Q\in E(\Qbar(C))$ be non-torsion points on the generic fiber $E$ of $\pi$. Assume there exist infinitely many $\l\in C(\Qbar)$ such that for some nonzero integers $m_{1,\l},m_{2,\l}$, we have that $[m_{i,\l}](P_{i,\l})=Q_\l$ on the $\l$-fiber $E_\l$ of $\pi$. Then   at least one of the following three conditions must hold:
\begin{itemize}
\item[(A)] there exists a nonzero integer $k$ such that for some $i\in\{1,2\}$, we have $[k](P_i)=Q$ on $E$. 
\item[(B)] there exist nonzero integers $k_1$ and $k_2$ such that $[k_1](P_1)=[k_2](P_2)$ on $E$.
\item[(C)] $E$ has complex multiplication by some order $\OO$ in an imaginary quadratic number field and  there exist nonzero $\alpha_1,\alpha_2,\beta\in \OO$ with $\alpha_1/\alpha_2\notin\Q$ such that $[\alpha_1](P_1)=[\alpha_2](P_2)=[\beta](Q)$ on $E$.
\end{itemize}
\end{proposition}

We note that Proposition~\ref{prop:direct} provides the direct implication in Theorem~\ref{thm:main} without the non-constant assumption for the corresponding sections $\sigma_{P_1},\sigma_{P_2},\sigma_Q$ (see also Remark~\ref{rem:constant}).

%
%

\begin{proof}[Proof of Proposition~\ref{prop:direct}.]
First, we note that we work under the hypotheses from Theorem~\ref{thm:main_0}. Therefore, if $E$ does not have complex multiplication, then  Theorem~\ref{thm:main_0} delivers alternatives~(A)~or~(B) from the conclusion of  Proposition~\ref{prop:direct}. 

So, from now on, we assume that $E$ has complex multiplication by some order $\OO$ in an imaginary quadratic field. As always, we identify $\End(E)$ with $\OO$ and for each $\gamma\in\OO$, we denote by $[\gamma]$ the corresponding endomorphism of $E$.

Now, if  alternatives~(A)-(B) do not hold in Proposition~\ref{prop:direct}, then we may assume from now on (due to Theorem~\ref{thm:main_0}~(ii)) that there exist nontrivial endomorphisms $\alpha_1,\alpha_2\in\OO$ such that
\begin{equation}
\label{eq:1}
[\alpha_1](P_1)=[\alpha_2](P_2).
\end{equation}
We may also assume that $\frac{\alpha_1}{\alpha_2}\notin\Q$ since otherwise (after multiplying $\alpha_1$ and $\alpha_2$ in equation~\eqref{eq:1} by a suitable nonzero $\gamma\in\OO$) we may assume $\alpha_1$ and $\alpha_2$ are nonzero integers and therefore, conclusion~(B) would be met. Furthermore, again multiplying by a suitable nonzero $\gamma\in\End(E)$, we may assume that $\alpha_1$ is a nonzero integer, while $\alpha_2\notin \Z$.
\begin{lemma}
\label{lem:not_both_integers}
Assume $\alpha_1$ is a nonzero integer and $\alpha_2\notin\Z$. Then there exist nonzero $\gamma,\beta\in\OO$ such that $[\gamma\alpha_1](P_1)=[\gamma\alpha_2](P_2)=[\beta](Q)$.
\end{lemma}

\begin{proof}[Proof of Lemma~\ref{lem:not_both_integers}.]
According to the hypotheses of Proposition~\ref{prop:direct}, there exist infinitely many $\l\in C(\Qbar)$ and there exist nonzero integers $n_{i,\l}$ (for $i=1,2$) such that 
\begin{equation}
\label{eq:20}
\left[n_{1,\l}\right]\left(P_{1,\l}\right) =\left[n_{2,\l}\right]\left(P_{2,\l}\right)=Q_\l. 
\end{equation}
In particular, we get
\begin{equation}
\label{eq:2}
\left[n_{1,\l}\cdot \alpha_1\right]\left(P_{1,\l}\right)=\left[n_{2,\l}\cdot \alpha_1\right]\left(P_{2,\l}\right).
\end{equation}
On the other hand, specializing the relation from equation~\eqref{eq:1} on the fiber corresponding to the point $\l\in C(\Qbar)$ (and then multiplying by $n_{1,\l}$) yields
\begin{equation}
\label{eq:3}
\left[n_{1,\l}\cdot \alpha_1\right]\left(P_{1,\l}\right)= \left[n_{1,\l}\cdot \alpha_2\right]\left(P_{2,\l}\right).
\end{equation}
Combining equations~\eqref{eq:2}~and~\eqref{eq:3} yields 
\begin{equation}
\label{eq:4}
\left[n_{2,\l}\cdot \alpha_1-n_{1,\l}\cdot \alpha_2\right]\left(P_{2,\l}\right)=0\text{ (on the corresponding fiber $E_\l$).}
\end{equation}
Now, because $\alpha_2\notin\Q$ (and $n_{1,\l}\ne 0$, while $\alpha_1,n_{2,\l}\in\Z$), we obtain that 
$$m_\l:=n_{2,\l}\cdot \alpha_1-n_{1,\l}\cdot \alpha_2\in\OO\text{ is nonzero,}$$  i.e., the corresponding endomorphism $[m_\l]$ of $E_\l$ is dominant. Therefore, equation~\eqref{eq:4} yields that $P_{2,\l}$ is a torsion point for the elliptic curve $E_\l$. Then knowing that $[n_{2,\l}](P_{2,\l})=Q_{\l}$ (see equation~\eqref{eq:20}) yields that also $Q_{\l}$ is a torsion point for $E_\l$. 

Now, the fact that there exist infinitely many points $\l\in C(\Qbar)$ such that both $P_{2,\l}$ and $Q_\l$ are torsion points for the elliptic curve $E_\l$ yields (see Theorem~\ref{thm:main_00}) that there exist nonzero  $\gamma,\delta\in \OO$ such that
\begin{equation}
\label{eq:5}
[\gamma](P_2)=[\delta](Q).
\end{equation}
Combining equations~\eqref{eq:5}~and~\eqref{eq:1} yields the desired conclusion~(C) from Proposition~\ref{prop:direct}:
$$\left[\alpha_1\gamma\right](P_1)=\left[\alpha_2\gamma\right](P_2) =[\alpha_2\delta](Q).$$
This concludes our proof of Lemma~\ref{lem:not_both_integers}.
\end{proof}

Lemma~\ref{lem:not_both_integers} finishes our proof of Proposition~\ref{prop:direct} by delivering the desired conclusion~(C).
\end{proof}


\section{Proof of the converse implication in Theorem~\ref{thm:main}}
\label{sec:converse}

In this Section, we prove the following main result.
\begin{proposition}
\label{prop:converse}
Let $C$ be a smooth projective curve defined over $\Qbar$, let $\mathcal{E}\lra C$ be an elliptic surface and let  $P_1,P_2,  Q\in E(\Qbar(C))$ be non-torsion points on the generic fiber $E$ of $\pi$.  Assume that  at least one of the following three conditions holds:
\begin{itemize}
\item[(A)] there exists a nonzero integer $k$ such that for some $i\in\{1,2\}$, we have $[k](P_i)=Q$ on $E$. 
\item[(B)] there exist nonzero integers $k_1$ and $k_2$ such that $[k_1](P_1)=[k_2](P_2)$ on $E$.
\item[(C)] $E$ has complex multiplication by some order $\OO$ in an imaginary quadratic number field and  there exist nonzero $\alpha_1,\alpha_2,\beta\in \OO$ with $\alpha_1/\alpha_2\notin\Q$ such that $[\alpha_1](P_1)=[\alpha_2](P_2)=[\beta](Q)$ on $E$.
\end{itemize}
Furthermore, assume that if $E$ is isotrivial, then not all three sections $\sigma_{P_1},\sigma_{P_2},\sigma_Q$ are constant. 
Then there exist infinitely many points $\l\in C(\Qbar)$ such that for some nonzero integers $m_{1,\l}$ and $m_{2,\l}$, we have
\begin{equation}
\label{eq:100}
\left[m_{1,\l}\right]\left(P_{1,\l}\right)=\left[m_{2,\l}\right]\left(P_{2,\l} \right)=Q_\l.
\end{equation}
\end{proposition}
Combining Propositions~\ref{prop:converse}~and~\ref{prop:direct}, we obtain the full result in Theorem~\ref{thm:main}.

We will prove Proposition~\ref{prop:converse} as a consequence of three other results: Lemmas~\ref{lem:A}~and~\ref{lem:B}, along with Proposition~\ref{prop:converse_main}; in each one of these three results, we  assume one of the conditions~(A)-(C) hold and then obtain the desired conclusion from equation~\eqref{eq:100}. 
We start by stating in Lemma~\ref{lem:solutions} a fact (proven in \cite[Theorem~A2]{Betti-4}), which will be used in our proofs.


\subsection{A useful result}
\label{subsec:very_useful}

We thank once again Umberto Zannier for explaining to us how to derive Lemma~\ref{lem:solutions} as a consequence of \cite[Theorem~A2]{Betti-4}.

\begin{lemma}
\label{lem:solutions}
Let $C$ be a smooth projective curve defined over $\Qbar$, let $\mathcal{E}\lra C$ be an elliptic surface and let  sections $\sigma_{P},\sigma_{Q}$ (corresponding to points $P,  Q$ of the generic fiber $E$) of $\mathcal{E}$; furthermore, assume $P$ is a non-torsion point of $E(\Qbar(C))$.  Also, if $\mathcal{E}$ is isotrivial, assume that not both sections $\sigma_P$ and $\sigma_Q$ are constant.

Then there exist infinitely many $\l\in C(\Qbar)$ such that for some nonzero integer $n_\l$, we have that  $[n_\l]\left(P_{\l}\right)=Q_{\l}$. 
\end{lemma}

The proof of Lemma~\ref{lem:solutions} is essentially contained in  \cite[Theorem~A2]{Betti-4}. We explain next the small differences between our statement and the one from \cite{Betti-4}; even though \cite[Theorem~A2]{Betti-4} is stated for non-isotrivial elliptic surfaces defined over $\R$, its proof contains all the necessary ingredients for the conclusion from Lemma~\ref{lem:solutions}.

\begin{proof}[Proof of Lemma~\ref{lem:solutions}.]
First, we consider the special case when $\mathcal{E}$ is isotrivial \emph{and} $\sigma_P$ is a constant section. In this case, due to our assumption, we have that $\sigma_Q$ is not constant. So, in this case (at the expense of replacing $C$ by a finite cover), we have $\mathcal{E}=E_0\times_{{\rm Spec}(\Qbar)}C$ for some elliptic curve $E_0$ defined over $\Qbar$; also, we have that $P\in E_0(\Qbar)$, which means that $P=P_\l$ for \emph{all} points $\l\in C(\Qbar)$. Also, due to our hypothesis, in this case we know that $Q$ is \emph{not} a constant point; hence, $Q$ is a generic point for $E_0$. Hence, for any integer $n$,  there exists some $\l\in C(\Qbar)$ such that $Q_\l=[n](P)$ on $E_0$ (as we identify, in this case, each fiber of $\mathcal{E}$ with $E_0$). Furthermore, since $P$ is non-torsion, the points $[n](P)$ are all distinct and thus, also the set of $\lambda$'s for which $Q_\l=[n](P)$ is infinite, as desired.
  
So, from now on, we assume that (exactly) one the following two conditions is met:
\begin{equation}
\label{eq:(1)}
\text{$\mathcal{E}$ is not isotrivial (and $P\in E(\Qbar(C))$ is a non-torsion point); or}
\end{equation}
\begin{equation}  
\label{eq:(2)}
\text{$\mathcal{E}$ is isotrivial, and $\sigma_P$ is not a constant section.}
\end{equation}

For each $\l\in C(\C)$ such that the fiber $E_\l$ is smooth, there exists a complex analytic uniformization map (see \cite[Chapter~6]{Silverman}) $\exp_\l:\C/\Gamma_\l\lra E_\l(\C)$ (where $\Gamma_\l$ is a lattice inside $\C$) with the property that
\begin{equation}
\label{eq:75}
\exp_\l(nz)=[n]\left(\exp_\l(z)\right)\text{ for each integer $n$.}
\end{equation}
Since the analytic map $\exp_\l:\C/\Gamma_\l\lra E_\l(\C)$ is a bijection, there is a locally analytic inverse of $\exp_\l(z)$ (i.e., the elliptic logarithm $\log_\l(z)$) in a small open subset of $E_\l(\C)$ (with respect to the induced complex topology on $E_\l(\C)$). So, our desired conclusion that $[n](P_\l)=Q_\l$ can be re-written as follows:
\begin{equation}
\label{eq:70}
n\cdot\log_\l(P_\l)-\log_\l(Q_\l)\in \Gamma_\l.
\end{equation}
We let $\omega_1(\l)$ and $\omega_2(\l)$ be generators for $\Gamma_\l$, which vary locally analytically (i.e., when $\l$ lives in a sufficiently small open subset of $C(\C)$).   
We write $\log_\l(P_\l)$ and $\log_\l(Q_\l)$ 
 as linear combinations with real coefficients (\emph{Betti coordinates}) of $\omega_1(\l),\omega_2(\l)$ (for more details on Betti coordinates and their applications, see \cite{Betti-2,Betti-1,Betti-3}). So, we write
$$\log_\l(P)=a_1(\l)\omega_1(\l)+a_2(\l)\omega_2(\l)\text{ and }\log_\l(Q)= b_1(\l)\omega_1(\l)+b_2(\l)\omega_2(\l),$$
where $a_1(\l),a_2(\l),b_1(\l),b_2(\l)$ are real analytic functions, where $\l$ is a point of $C(\C)$, which varies in a complex small disk denoted $D$. We write $A(\l):=(a_1(\l),a_2(\l))$ and $B(\l):=(b_1(\l),b_2(\l))$. So, equation~\eqref{eq:70} can be re-stated (on $D$) as follows: 
\begin{equation}
\label{eq:71}
[n](P_\l)=Q_\l\text{ if and only if }n\cdot A(\l)-B(\l)\in\Z\times\Z.
\end{equation}
Now, viewing $D$ as an open subset of $\R^2$, we can compute the Jacobian matrix of the function $A=(a_1,a_2):D\lra \R^2$. In both cases~\eqref{eq:(1)}~and~\eqref{eq:(2)} above, the Jacobian matrix of $A$ has rank $2$ generically; when $\mathcal{E}$ is non-isotrivial, this is proven in \cite[Chapter~3]{Umberto}, while if $\mathcal{E}=E_0\times_{{\rm Spec}(\Qbar)}C$, this follows from the fact that $P$ is a generic point for $E_0$. Hence, at the expense of shrinking $D$, we may assume the Jacobian matrix of $A$ has rank $2$ on the entire set $D$. Then for large $n$, we have that
$$F_n(\l):=A(\l)-\frac{B(\l)}{n}$$
is real analytic on $D$ and also its Jacobian matrix has rank $2$. Arguing identically as in \cite[Theorem~A2]{Betti-4}, one obtains that for large $n$,  $F_n(D)$  contains a nonempty open set $V\subset \R^2$, which is independent of $n$. Clearly,  for large $n$, $V$ contains rational points whose coordinates have denominators dividing $n$. More precisely, there exist at least $cn^2$ (for some positive real number $c$, depending on $V$, but independent of $n$) suitable $\l\in D$ such that  $nF_n(\l)\in\Z\times\Z$. Then equation~\eqref{eq:71} yields that $[n](P_\l)=Q_\l$ for each of these $cn^2$ points $\l\in C$. As $n$ can be chosen arbitrarily large, we obtain the desired conclusion in Lemma~\ref{lem:solutions}.
\end{proof}

\begin{remark}
\label{rem:primary}
When $Q=0\in E$ and $E$ is an elliptic curve defined over a number field $K$, a stronger version of Lemma~\ref{lem:solutions} was studied in the context of \emph{primary primes} for an elliptic divisibility sequence. So, one views $E$ as the generic fiber of an elliptic arithmetic surface $\pi:\mathcal{E}\lra {\rm Spec}(\OO_K)$ (where $\OO_K$ is the ring of algebraic integers in $K$). Also, for any given point $R\in E(K)$, one considers the corresponding section $\sigma_R$ of $\pi$ and for a prime $v$ of $\OO_K$ (of good reduction), we denote by $R_v:=\sigma_R(v)$ the corresponding specialization of $R$ at the place $v$. Then it was proven (see \cite{Patrick, Patrick-2}) that for all sufficiently large positive integers $n$, there exists a (nonarchimedean) place $v_n$ of $K$ (such that the fiber $\pi^{-1}(v_n)=:E_{v_n}$ is smooth) with the property that on $E_{v_n}$ we have:
\begin{equation}
\label{eq:primary}
[n]\left(P_{v_n}\right)=0\text{ and moreover, }[m]\left(P_{v_n}\right)\ne 0\text{ when $0<m<n$.}
\end{equation}
The existence of primary primes for all sufficiently large positive integers $n$ is a \emph{stronger} conclusion than the fact that there are infinitely many  primes $v$ of $\OO_K$ for which there exists \emph{some} positive integer $n$ such that $[n](P_v)=0$. 

Verzobio \cite{Verzobio} extended the study of primary primes by considering a target point $Q\ne 0$. Under the hypothesis that the target point $Q$ is dynamically related to the starting point $P$ (i.e., $\Phi(P)=\Psi(Q)$ for some nontrivial endomorphisms $\Phi,\Psi\in\End(E)$), Verzobio \cite{Verzobio} proved that again for all sufficiently large positive integers $n$, there exists a prime $v_n$ of $\OO_K$ such that
\begin{equation}
\label{eq:primary-2}
[n]\left(P_{v_n}\right)=Q_{v_n}\text{ and moreover, }[m]\left(P_{v_n}\right)\ne Q_{v_n}\text{ when $0<m<n$.}
\end{equation}
Furthermore, Matteo Verzobio confirmed to us that his method extends almost verbatim to cover the function field case (i.e., when $E$ is an elliptic curve defined over $\Qbar(C)$, which is the setting of Lemma~\ref{lem:solutions}); the necessary ingredients for modifying the argument from \cite{Verzobio} to the function field case are contained in the papers \cite{Bartosz, Slob}.
\end{remark}


\subsection{Proof of Proposition~\ref{prop:converse} assuming conditions~(A)~or~(B) hold}
The following result is an easy consequence of Lemma~\ref{lem:solutions}.

\begin{lemma}
\label{lem:A}
The conclusion in Proposition~\ref{prop:converse} holds assuming condition~(A) is satisfied. 
\end{lemma}

\begin{proof}
Without loss of generality, assume there exists a nonzero integer $k$ such that $[k](P_1)=Q$; in particular, this means that for \emph{all} points $\l\in C(\Qbar)$ (for which the corresponding fiber $E_\l$ is an elliptic curve), we have $[k](P_{1,\l})=Q_\l$. On the other hand, Lemma~\ref{lem:solutions} yields infinitely many points $\l\in C(\Qbar)$ along with nonzero integers $n_\l$ such that $[n_\l](P_{2,\l})=Q_\l$. 

Note that if $\mathcal{E}$ is isotrivial, we do \emph{not} have that both sections $\sigma_{P_2}$ and $\sigma_Q$ are constant since otherwise, we would get that also $\sigma_{P_1}$ is constant (since $[k](P_1)=Q$ for a nonzero integer $k$), which contradicts one of the hypotheses in Proposition~\ref{prop:converse}. Thus, indeed, we can apply Lemma~\ref{lem:solutions} to the triple $(\mathcal{E},P_2,Q)$. This concludes our proof of Lemma~\ref{lem:A}.
\end{proof}

\begin{lemma}
\label{lem:B}
The conclusion in Proposition~\ref{prop:converse} holds assuming condition~(B) is satisfied. 
\end{lemma}

\begin{proof}
So, we know there exist nonzero integers $k_1$ and $k_2$ such that $[k_1](P_1)=[k_2](P_2)$. Arguing as in the proof of Lemma~\ref{lem:A}, we obtain that the hypotheses from Lemma~\ref{lem:solutions} are met by the triple $(\mathcal{E},[k_1](P_1),Q)$. Indeed, if $\mathcal{E}=E_0\times_{{\rm Spec}(\Qbar)}C$ (for some elliptic curve $E_0$ defined over $\Qbar$) and both $[k_1](P_1)$ and $Q$ were points in $E_0(\Qbar)$, then also $P_1,P_2\in E_0(\Qbar)$ (note that $[k_1](P_1)=[k_2](P_2)$ and $k_1k_2\ne 0$).

Now, applying Lemma~\ref{lem:solutions} to the points $[k_1](P_1)$ and $Q$ yields the existence of infinitely many points $\l\in C(\Qbar)$ along with nonzero integers $n$ such that 
\begin{equation}
\label{eq:11}
[nk_1]\left(P_{1,\l}\right)=Q_\l.
\end{equation}
Equation~\eqref{eq:11} yields that also $[nk_2](P_{2,\l})=Q_\l$, thus providing the desired conclusion in Proposition~\ref{prop:converse}.
\end{proof}

So, in order to complete the proof of Proposition~\ref{prop:converse}, it suffices to prove the conclusion~\eqref{eq:100} using the hypothesis~(C).


\subsection{Proof of Proposition~\ref{prop:converse} assuming condition~(C) holds} So, $E$ is a $CM$ elliptic curve whose endomorphism ring is identified with an order $\OO=\ZZ[\theta]$ in an imaginary quadratic field (see equation~\eqref{eq:theta} for the definition of $\theta$).

\begin{proposition}
\label{prop:converse_main}
With the notation as in Proposition~\ref{prop:converse}, assume in addition that $E$ has complex multiplication by $\OO=\Z[\theta]$. Also, we assume there exist nonzero $\alpha_1,\alpha_2,\beta\in\OO$ with $\alpha_1/\alpha_2\notin\Q$ such that 
\begin{equation}\label{eq:relations}
  [\alpha_1](P_1)=[\alpha_2](P_2)=[\beta](Q) \quad\text{ (in $E$).}
\end{equation}
Then there exist infinitely many $\lambda\in C(\QQbar)$ with the property that there exist nonzero integers $m_{1,\l},m_{2,\l}$  such that
\[
  \left[m_{1,\l}\right]\left(P_{1,\lambda}\right)= \left[m_{2,\l}\right]\left(P_{2,\lambda}\right)=Q_\lambda.
\]
\end{proposition}

\begin{proof}
First, we note that since $E$ has $CM$, then its $j$-invariant must be contained in $\Qbar$ (see \cite[Chapter~2]{Silverman-2}), i.e., $\mathcal{E}$ is isotrivial. So, the hypothesis from Proposition~\ref{prop:converse} yields that at least one of the sections $\sigma_{P_1},\sigma_{P_2},\sigma_Q$ is not constant. However, due to relation~\eqref{eq:relations}, we get that \emph{none} of these three sections is constant.

Let $M:=N(\alpha_1)\cdot N(\alpha_2)\cdot N(\beta)$, where $N(\cdot )$ is the usual norm for $\Q(i\sqrt{D})$ (see equation~\eqref{eq:norm}). Let $\ell\in\{1,\dots, M\}$ satisfying the conclusion of Lemma~\ref{lem:arithm_2}, i.e., for each (odd) integer $a$ satisfying $a\equiv 2\ell-1\pmod{2M}$, we have that 
\begin{equation}
\label{eq:40}
\gcd\left(N\left(a+4\theta\right),2M\right)=1.
\end{equation}
\begin{lemma}
\label{lem:solutions_2}
There exist infinitely many $\l\in C(\Qbar)$ with the property that there exist some positive integer $a_\l$ such that $a_\l\equiv 2\ell-1\pmod{2M}$ and $[a_\l+4\theta](Q_\l)=0$. 
\end{lemma}

\begin{proof}[Proof of Lemma~\ref{lem:solutions_2}.]
Letting $a_\l=2Mn_\l + (2\ell-1)$ (for some $n_\l\in \N$), we can re-write the conclusion from Lemma~\ref{lem:solutions_2} as 
\begin{equation}
\label{eq:50}
[n_\l]\left([2M](Q_\l)\right)=[1-2\ell-4\theta](Q_\l).
\end{equation} 
The desired conclusion from Lemma~\ref{lem:solutions_2} follows from Lemma~\ref{lem:solutions} applied to the starting point $[2M](Q)$ and with the target point $[1-2\ell-4\theta](Q)$ (note the hypotheses from Lemma~\ref{lem:solutions} are met because $\sigma_Q$ is not a constant section). 
\end{proof}

Now, let a point $\l\in C(\Qbar)$ along with a suitable positive integer $a_\l$ as in the conclusion of Lemma~\ref{lem:solutions_2}; we let $U_\l$ be the cyclic subgroup of $E_\l$ generated by $Q_\l$.  Corollary~\ref{cor:action},~part~(i) yields that for any $\gamma\in\OO$, we have a natural induced group homorphism $\widetilde{[\gamma]}:U_\l\lra U_\l$ given by $Q\mapsto [\gamma](Q)$. Furthermore, part~(ii) of Corollary~\ref{cor:action} yields that $\widetilde{[\beta]}$, $\widetilde{[\bar{\alpha_1}]}$ and $\widetilde{[\bar{\alpha_2}]}$ are group automorphisms of $U_\l$. 
\begin{lemma}
\label{lem:multiples}
With the above notation, we have that 
\begin{equation}
\label{eq:51}
[N(\alpha_1)]\left(P_{1,\l}\right)\text{, }[N(\alpha_2)]\left(P_{2,\l}\right) \in U_\l.
\end{equation}
Moreover, both $[N(\alpha_1)](P_{1,\l})$ and $[N(\alpha_2)](P_{2,\l})$ are generators for the cyclic group $U_\l$.
\end{lemma}  

\begin{proof}[Proof of Lemma~\ref{lem:multiples}.]
Let $i\in\{1,2\}$. Equation~\eqref{eq:relations} yields that $[\alpha_i](P_{i,\l})\in U_\l$ and therefore, also $[\bar{\alpha_i}\cdot \alpha_i](P_{i,\l})\in U_\l$; thus, we obtain the conclusion from equation~\eqref{eq:51}. Furthermore, we have the more precise relation coming from equation~\eqref{eq:relations}:
\begin{equation}
\label{eq:53}
\left[N(\alpha_i)\right]\left(P_{i,\l}\right)=\left[\bar{\alpha_i}\cdot \beta\right]\left(Q_\l\right).
\end{equation}
Because the induced maps $\widetilde{[\bar{\alpha_i}]}$ and $\widetilde{[\beta]}$ are group automorphisms of $U_\l$, then we obtain that $[N(\alpha_i)](P_{i,\l})$ is a generator for the cyclic group $U_\l$ (for each $i=1,2$).

This concludes our proof of Lemma~\ref{lem:multiples}.
\end{proof}

Lemma~\ref{lem:multiples} yields the existence of positive integers $k_{1,\l}$ and $k_{2,\l}$ such that
\begin{equation}
\label{eq:52}
\left[k_{1,\l}N(\alpha_1)\right]\left(P_{1,\l}\right)=\left[k_{2,\l}N(\alpha_2) \right]\left(P_{2,\l}\right)=Q_\l.
\end{equation}
Equation~\eqref{eq:52} along with the fact that there exist infinitely many such points $\l\in C(\Qbar)$ (according to Lemma~\ref{lem:solutions_2}) deliver the desired conclusion in Proposition~\ref{prop:converse_main}.
\end{proof}

As previously explained, Lemmas~\ref{lem:A}~and~\ref{lem:B}, along with Proposition~\ref{prop:converse_main} finish our proof of Proposition~\ref{prop:converse}.


\end{document}